\documentclass[12pt]{article}
\usepackage{amsmath, amsfonts, amssymb, amsthm,harvard}
\usepackage[ps2pdf,linktocpage,bookmarks,bookmarksnumbered,bookmarksopen]{hyperref}
\hypersetup{%
   pdftitle   ={A short note on small deviations of sequences of i.i.d. random variables with exponentially decreasing weights},
   pdfsubject ={Manuscript},
   pdfauthor  ={Frank Aurzada},
   pdfkeywords={Small deviation; lower tail probability; Laplace transform; sums of independent random variables.}
      }

\let\BFseries\bfseries\def\bfseries{\BFseries\mathversion{bold}} 
\newcommand{\ind}{1\hspace{-0.098cm}\mathrm{l}}
\newcommand{\eps}{\varepsilon}
\theoremstyle{plain}
\newtheorem{thm}{Theorem}
\newtheorem{lem}[thm]{Lemma}

\theoremstyle{definition}

\newtheorem{rem}[thm]{Remark}
\newtheorem{exa}[thm]{Example}
\renewenvironment{proof}[1][] {\noindent {\bf Proof#1:} }{\hspace*{\fill}$\square$\medskip\par}
\def\P{{\bf {\mathbb{P}}}}
\newcommand{\pr}[1]{\P\left(#1\right)}

\def\E{\mathbb{E}} 

\newcommand{\indi}[1]{\,\ind_{\{#1\}}}
\def\R{\mathbb{R}}\def\N{\mathbb{N}} \def\d{\mathrm{d}}

\newcommand{\deq}{\stackrel{d}{=}}
\def\wast{{X}}
\renewcommand{\harvardand}{and}

\begin{document}
\title{A short note on small deviations of sequences of i.i.d.\ random variables with exponentially decreasing weights}
\author{Frank Aurzada\footnote{Address: Technische Universit\"{a}t Berlin, Institut f\"{u}r Mathematik, Sekr.\ MA 7-5, Str.\ des 17.\ Juni 136, 10623 Berlin, Germany. Email: {\tt aurzada@math.tu-berlin.de}}}
\date{\today}
\maketitle

\begin{abstract} We obtain some new results concerning the small deviation problem for $S=\sum_n q^n X_n$ and $M=\sup_n q^n X_n$, where $0<q<1$ and $(X_n)$ are i.i.d.\ non-negative random variables. In particular, the asymptotics is shown to be the same for $S$ and $M$ in some cases.\end{abstract}



\bigskip
\begin{center} To appear in: {\it Statistics and Probability Letters} \end{center}

\bigskip
\section{Introduction}
Let us consider the following series
\begin{equation}S=\sum_{n=0}^\infty \sigma_n \wast_n,\label{eqn:s}\end{equation}
where $(\sigma_n)$ is a sequence of positive numbers with $\sigma_n\to 0$ and $\wast, \wast_0, \wast_1, \ldots$ are a.s.\ non-negative, i.i.d.\ random variables. By Kolmogorov's Three Series Theorem, $S<\infty$ a.s.\ if and only if \begin{equation}\sum_{n=0}^\infty \E \min(1,\sigma_n \wast) < \infty. \label{eqn:3series} \end{equation} 
Under this condition, we study the small deviation problem, also called small ball problem, or lower tail probability problem for $S$: \begin{equation}\pr{ S \leq \eps},\qquad \text{as $\eps\to 0$.} \label{eqn:p}\end{equation}

There has been a lot of interest in small deviation problems in recent years (cf.\ the surveys by \citeasnoun{lif} and \citeasnoun{lishao}). The question of small deviations of sums was first addressed by \citeasnoun{sytaya}, \citeasnoun{gausslpearly}, and \citeasnoun{ligausslp} for Gaussian random variables. Later, this was generalized by \citeasnoun{lifap} (see also references therein), \citeasnoun{dll}, and \citeasnoun{roz} for random variables $\wast$ with finite variance. Then \citeasnoun{phd}, \citeasnoun{aurzada1}, and, most recently, \citeasnoun{br1} and \citeasnoun{br2}, give general treatments only under the necessary and sufficient condition (\ref{eqn:3series}).

The case of {\em polynomial} decrease of $(\sigma_n)$ seems to be almost completely understood by now thanks to \citeasnoun{br1}. The latter paper shows that there are different regimes in the sense that if $(\sigma_n)$ decreases slowly enough (in relation to the lower tail of $\wast$) then that order of decrease determines the lower tail of $S$. If, on the other hand, $(\sigma_n)$ decreases sufficiently fast then the lower tail of $\wast$ determines the lower tail of $S$.

In this note, we study the case of {\em exponential} decrease of $(\sigma_n)$, which is considered by \citeasnoun{dll} under the assumption $\E \wast^2<\infty$ and $\wast$ absolutely continuous, and with similar conditions to ours by \citeasnoun{br2}. Here we continue the theme from \citeasnoun{aurzada1}, \citeasnoun{br1}, and \citeasnoun{br2} and investigate (\ref{eqn:p}) under minimal assumptions. This must necessarily yield less precise assertions. In particular, we confine ourselves to the logarithmic order of the small deviations. Contrary to this, the early works for this problem mentioned above focussed on solving (\ref{eqn:p}) for particular distributions or under restrictive assumptions, such as finite variance, which could then yield very precise results such as the strong asymptotic order of (\ref{eqn:p}).

We are going to use the following notation for strong and weak asymptotics: We write $f\lesssim g$, if $\limsup f/g \leq 1$. Analogously, $f\gtrsim g$ is defined. Furthermore, $f\sim g$ means $\lim f/g = 1$. We also use $f\approx g$ if $0<\liminf f/g \leq \limsup f/g < \infty$.

As already mentioned, we consider the special case that $\sigma_n \sim q^n$ with some $0<q<1$. In order to get the logarithmic order of (\ref{eqn:p}) it is, in most cases, sufficient to treat $\sigma_n = q^n$. We come back to this question in Remark~\ref{rem:simplif}. I.e.\ we consider
\begin{equation}S=\sum_{n=0}^\infty q^n \wast_n.\label{eqn:s1}\end{equation}
In view of (\ref{eqn:3series}), we need to assume that the necessary and sufficient condition for the problem to be well-posed, \begin{equation}
\E \log\max(\wast,1) < \infty, \label{eqn:starns}
\end{equation}
holds, which we do in the following. This is the only condition imposed on the upper tail of $\wast$ in this paper. Note that this condition does not depend on $q$, i.e.\ on the sequence $(\sigma_n)$, unlike in the polynomial case, cf.\ \citeasnoun{phd}, \citeasnoun{aurzada1}, or \citeasnoun{br1}.

The main argument in this note (presented in Section~\ref{sec:main}) allows to calculate the logarithmic small deviation order and constant for $S$ from (\ref{eqn:s1}) for basically all commonly considered distributions. Furthermore, this can be extended to determine the small deviation order of the supremum $M$ given by $M=\sup_n q^n \wast_n$, which surprisingly leads to the same rate in many cases.

Apart from presenting the new results related to the small deviations of $S$ and $M$ with exponentially decreasing weights, this note has the purpose of giving a simple illustration of the general treatment of the small deviations of $S$ from (\ref{eqn:s}). Namely, to the knowledge of the author this case is the easiest example that shows how the problem is tackled, in particular, how fruitful the use of Tauberian theorems is in this context.

The paper is organized as follows. In Section~\ref{sec:main}, we present the strikingly simple main argument that helps us to solve the small deviation problem for $S$ from (\ref{eqn:s1}). In Section~\ref{sec:app}, several examples are given under various conditions for the lower tail of $\wast$. We continue in Section~\ref{sec:sup} with a comparison to the small deviations of the supremum $M$. Section~\ref{sec:final} provides the necessary Tauberian-type arguments and a final comment.

\section{Main argument} \label{sec:main}
The main argument is very simple. Recall that considering small deviations is equivalent to considering Laplace transforms at infinity, by the use of Tauberianian-type theorems, cf.\ \citeasnoun{bgt} and Lemma~\ref{lem:app} below. Let us define $$F(\lambda):=-\log \E e^{-\lambda S}\qquad\text{and}\qquad G(\lambda):=-\log \E e^{-\lambda \wast}.$$ The goal is to deduce the behaviour of $F$ from the behaviour of $G$.

Note that in this particular case, (\ref{eqn:s1}) yields \begin{equation}
q S + \wast \deq S. \label{eqn:trick}
\end{equation} This implies, by the independence of the $\wast_n$'s, $F(q \lambda) + G(\lambda) = F(\lambda)$. We iterate this and obtain that, for all $N\geq 0$, $$ F(\lambda) = G(\lambda) +G(q \lambda) + \ldots + G(q^{N} \lambda)+ F(q^{N+1} \lambda).$$
Note that, when $\lambda>1$, one can choose $N=N(\lambda)$ such that $q^{N+1} \lambda <  1 \leq q^{N}\lambda$. Then $0\leq F(q^{N+1} \lambda)\leq F(1)$. We summarise:

\begin{lem} \label{lem:thelem} For any $\lambda> 1$, \begin{equation} \sum_{0\leq n \leq \frac{\log \lambda}{\log 1/q}} G(q^n \lambda) \leq F(\lambda) \leq F(1)+ \sum_{0\leq n \leq \frac{\log \lambda}{\log 1/q}} G(q^n \lambda). \label{eqn:fg2}\end{equation} \end{lem}

Note that $F(1)$ is independent of $\lambda$. Thus, one can determine the asymptotics of $F$ at infinity completely from the behaviour of $G(\lambda)$ for $\lambda\geq 1$. We are going to apply this lemma in several examples.

For concrete distributions, $G$ can be calculated explicitly; and Lemma~\ref{lem:thelem} turns the small deviation problem into an easy calculation exercise. If only the asymptotics of the lower tail of the distribution of $\wast$ is given, one can obtain the behaviour of $G$ at infinity, and thus the one of $F$ via (\ref{eqn:fg2}).

Recall from \citeasnoun{aurzada1} that in the case of polynomially decreasing sequences $(\sigma_n)$ in (\ref{eqn:s}), it is possible to treat the small deviation problem for $S$ analogously to the small deviation problem for $M=\sup_n \sigma_n \wast_n$. Note that this is not possible with the current approach, since the technique relies on (\ref{eqn:trick}). However, we come to the supremum case in Section~\ref{sec:sup}.

We finally point out that Lemma~\ref{lem:thelem} does not require any other condition apart from the necessary and sufficient condition for the problem to make sense, (\ref{eqn:starns}), which ensures the finiteness of all expressions.

\section{Applications} \label{sec:app}
\subsection{Random variables with positive mass at the origin}
We start with the case when $\wast$ has a positive mass at the origin. This corresponds to $G$ tending to a constant.

\begin{thm} Let $\pr{ \wast =0} =p_0>0$. Then \begin{equation}-\log \pr{S \leq \eps} \sim \frac{\log 1/p_0}{\log 1/q} \,\log 1/\eps ,\qquad \text{as $\eps\to 0$.} \label{eqn:sdd}\end{equation} \label{thm:a3} \end{thm}

\begin{proof} If $\wast$ has positive mass at zero $G(\lambda)\sim - \log p_0$. Note that in particular, $G(\lambda) \leq - \log p_0$ for all $\lambda$. This together with (\ref{eqn:fg2}) shows that \begin{multline*}F(\lambda)-F(1) \leq \sum_{0\leq n \leq \frac{\log \lambda}{\log 1/q}} G(q^n \lambda) \leq \sum_{0\leq n \leq \frac{\log \lambda}{\log 1/q}} - \log p_0 \\ \leq  \left(\frac{\log \lambda}{\log 1/q}+1 \right)(\log 1/p_0) \lesssim \frac{\log 1/p_0}{\log 1/q} \log \lambda.\end{multline*}

The proof of the lower bound is as follows. First, let us define $\tilde{G}$ by $G(\lambda)=:-\log p_0 + \tilde{G}(\lambda)$. Note that $\tilde{G}$ is non-positive and increasing and that $\tilde{G}(\lambda)\to 0$. Therefore, \begin{multline*}\sum_{0\leq n \leq \frac{\log \lambda}{\log 1/q}} \tilde{G}(q^n \lambda) =  \sum_{0\leq n \leq \frac{\log \lambda}{\log 1/q}} \frac{\tilde{G}(q^n \lambda)}{q^{n} (1-q)}\, \int_{q^{n+1}}^{q^{n}} \d x \\ \geq   \sum_{0\leq n \leq \frac{\log \lambda}{\log 1/q}} \int_{q^{n+1}}^{q^{n}} \frac{\tilde{G}(x \lambda)}{x (1-q)}\,  \d x \geq \int_{q/\lambda}^{1} \frac{\tilde{G}(x \lambda)}{x (1-q)}\,  \d x = \int_{q}^{\lambda} \frac{\tilde{G}(y)}{y(1-q)} \,  \d y.\end{multline*} Now l'Hospital's Rule shows that the last term divided by $\log \lambda$ tends to zero, since $\tilde{G}(\lambda)\to 0$. This yields that $F(\lambda)\gtrsim (\log \lambda) \log (1/p_0) / \log (1/q) $. Thus, we finally obtain that $F(\lambda) \sim (\log \lambda) \log (1/p_0) / \log (1/q)$, and the assertion follows from Lemma~\ref{lem:app}.
\end{proof}

\begin{exa} As an illustrative example consider Bernoulli random variables given by $\pr{ \wast =0}=\pr{ \wast =1}=1/2=p_0$ and $q=1/2$. In this case, $S$ is uniformly distributed in $[0,1]$ and equality holds for all $0<\eps<1$ in (\ref{eqn:sdd}).

However, if we take for $\wast$ a distribution that has sufficiently heavy tails and a mass at zero, none of the results in the literature applies. \end{exa}

\subsection{Random variables with regularly varying lower tail}
Now, we continue with the result under the assumption that the distribution function of $\wast$ behaves essentially as a polynomial for small values. This corresponds to the case of logarithmic increase of $G$.

\begin{thm} Let $\pr{ \wast \leq \eps} \approx \eps^\beta$, as $\eps\to 0$, for some $\beta>0$. Then \begin{equation}-\log \pr{S \leq \eps} \sim \frac{\beta (\log 1/\eps)^2}{2 \log 1/q},\qquad \text{as $\eps\to 0$.} \label{eqn:addsvt}\end{equation} \label{thm:a1} \end{thm}

\begin{proof} The assumption for the lower tail of $\wast$ implies that $\E e^{-\lambda \wast} \approx \lambda^{-\beta}$, as $\lambda\to \infty$. 
The monotonicity of $G$ implies that there are constants $C_1, C_2\in \R$ such that $$\beta \log \lambda +C_1 \leq G(\lambda) \leq \beta \log \lambda +C_2,\qquad\text{for all $\lambda\geq 1$.}$$ From ({\ref{eqn:fg2}}) we deduce that $$F(\lambda) \leq F(1)+ \sum_{0\leq n \leq \frac{\log \lambda}{\log 1/q}} (\beta \log (q^n \lambda) +C_2).$$
After a short calculation one can see that the strong asymptotic order of the right hand side, when $\lambda\to\infty$, is $\beta (\log \lambda)^2/(2 \log 1/q)$.

Analogous arguments show the lower bound. I.e.\ we obtain $$F(\lambda) \sim \frac{\beta (\log \lambda)^2}{2 \log 1/q},\qquad \text{as $\lambda\to\infty$.}$$ Finally we use Lemma~\ref{lem:app}. 
\end{proof}

This theorem is proved by \citeasnoun{br2}; however, the present method uses a significantly simpler proof. Also, a stronger version of this theorem was proved by \citeasnoun{br2} for rational $\beta$ and with a stronger assumption for the lower tail of $\wast$.

\begin{rem} With essentially the same proof, one can show even more. Namely, assume that $\pr{ \wast \leq \eps} \approx \eps^\beta \ell(1/\eps)$, as $\eps\to 0$, where $\beta>0$ and $\ell$ is a slowly varying function at infinity that is bounded away from $0$ and $\infty$ on every compact subset of $[1,\infty[$. Then (\ref{eqn:addsvt}) holds. In particular, $\ell$ has no influence on the order, cf.\ \citeasnoun{br2}. \end{rem}

\begin{exa} As main examples one can consider $\wast=|\wast'|^p$ for Gaussian random variables $\wast'$ or symmetric stable random variables $\wast'$ (in both of which cases $\beta=1/p$). One can as well consider $\wast$ with a Gamma distribution, Weibull distribution, etc.

Recall that \citeasnoun{dll} consider the most important case $\wast'$ standard normal and $p=2$ and give much more precise estimates. However, for random variables with a heavy tail (e.g.\ the mentioned stables) the above result seems to be new. \end{exa}

\begin{rem} Using essentially the same technique as in the proof of Theorem~\ref{thm:a1} one also tackles the case when $-\log \pr{X\leq \eps}$ is some other slowly varying function. For example, in the case that $X$ is log normal one obtains in this way that $-\log \pr{S \leq \eps} \sim (\log 1/\eps)^3 / (6 \log 1/q)$, as $\eps\to 0$. \end{rem}

\subsection{Random variables with exponentially small lower tail}
Finally, we consider the case that $\wast$ has exponentially little mass near the origin. This corresponds to the case of polynomial increase of $G$.

\begin{thm} Let $-\log \pr{ \wast \leq \eps} \sim  K \eps^{-\gamma}$, as $\eps\to 0$, for some $\gamma>0$ and $K>0$. Then \begin{equation} -\log \pr{S \leq \eps} \sim  \frac{{K}}{(1-q^{{\gamma/(1+\gamma)}})^{1+\gamma}}\, \eps^{-\gamma},\qquad \text{as $\eps\to 0$.}\label{eqn:mexca1} \end{equation} \label{thm:a2} \end{thm}

\begin{proof} From the assumption we deduce, by \citeasnoun{bgt}, Theorem~4.12.9, that $G(\lambda) \sim {K'} \lambda^{\gamma'}$, as $\lambda\to \infty$, where $K=({K'}{\gamma'}^{\gamma'} (1-{\gamma'})^{1-{\gamma'}})^{1/(1-{\gamma'})}$ and $\gamma={\gamma'}/(1-{\gamma'})$.

From ({\ref{eqn:fg2}}) we deduce in a similar fashion as in the proof of Theorem~\ref{thm:a1} that $$F(\lambda) \sim \frac{{K'} \lambda^{\gamma'}}{1-q^{\gamma'}},\qquad \text{as $\lambda\to\infty$.}$$ Again by \citeasnoun{bgt}, Theorem~4.12.9, we conclude the proof. \end{proof}

A similar result holds if we add a slowly varying term in the asymptotics of the lower tail of $\wast$. Furthermore, one can prove an analogous result under the assumption that $-\log \pr{ \wast \leq \eps} \approx \eps^{-\gamma}$.

Under only slightly stronger assumptions \citeasnoun{br1} obtained the same result.

\begin{exa} An example that fits perfectly into the situation of Theorem~\ref{thm:a2} is the inverse Weibull distribution, that has exponentially little mass near the origin and a heavy tail.

As another illustrative example consider a stable totally skewed random variable $\wast'$, i.e.\ a positive random variable with Laplace transform $\E e^{-\lambda \wast'} = \exp( - K \lambda^\alpha)$, with $K>0$ and $0<\alpha<1$, cf.\ \citeasnoun{ST}. These random variables are $\alpha$-stable and have exponentially little mass near the origin. Theorem~\ref{thm:a2} can be applied to $\wast={\wast'}^p$, where $\gamma=\alpha/(p(1-\alpha))$. Note that in the case $p=1$ the sum $S$ is also a stable random variable, which makes it easy to verify the constant occuring in (\ref{eqn:mexca1}). \label{exa:stable} \end{exa}

\section{Comparison with the supremum case} \label{sec:sup}
In this section, we give an extension and comparison to the small deviations of the supremum. Namely, we consider the small deviation problem for $$M:=\sup_n \sigma_n \wast_n,$$ when $\sigma_n=q^n$. Surprisingly, we obtain the same small deviation rate for $M$ and $S$ in the setup considered in Section~\ref{sec:app}. This contrasts the results for the cases when $(\sigma_n)$ is polynomially decreasing, where the small deviation rates for $M$ and $S$ are always distinct, cf.\ \citeasnoun{aurzada1}.

Let us start with the case when $\wast$ has a positive mass at the origin considered in Theorem~\ref{thm:a3}. The resulting small deviation rate and constant are the same for $S$ and $M$.

\begin{thm} Let $\pr{ \wast =0} =p_0>0$. Then $$-\log \pr{M \leq \eps} \sim -\log \pr{S \leq \eps} \sim \frac{\log 1/p_0}{\log 1/q} \,\log 1/\eps ,\qquad \text{as $\eps\to 0$.}$$ \label{thm:a3sup} \end{thm}

\begin{proof} Since $M\leq S$, we only have to show an upper bound for $\pr{M \leq \eps}$.

Let $\tau>1$. Since the distribution function $\pr{\wast\leq \delta}$ is c\`{a}dl\`{a}g, there is a $\delta>0$ such that $\pr{\wast\leq \delta} \leq \tau p_0$.

Let $\eps<\delta$ and choose the maximal $N\in\N$ such that $\eps q^{-N}\leq\delta$, i.e.\ asymptotically, when $\eps\to 0$, $N\sim \log (\delta/\eps)/\log (1/q) \sim \log (1/\eps)/\log (1/q)$. Then, by the independence of the $\wast_n$'s, \begin{multline*} \log \pr{M \leq \eps} \leq \sum_{n=0}^N \log\pr{\wast \leq \eps q^{-n}} \leq \sum_{n=0}^N \log\pr{\wast \leq \delta} \\ \leq \sum_{n=0}^N \log \tau p_0 = (N+1)\log \tau p_0. \end{multline*} If we let $\eps$ tend to zero, we obtain 
$$\limsup_{\eps\to 0} \frac{\log \pr{M \leq \eps}}{\log (1/\eps)/\log (1/q)} \leq \log \tau p_0,$$ for all $\tau>1$. We let $\tau\to 1$ to finish the proof. \end{proof}

We continue with the case considered in Theorem~\ref{thm:a1}. Here again we find that $M$ and $S$ have the same small deviation rate and constant.

\begin{thm} Let $\pr{ \wast \leq \eps} \approx \eps^\beta$, as $\eps\to 0$, for some $\beta>0$. Then $$-\log \pr{M \leq \eps}\sim -\log \pr{S \leq \eps} \sim \frac{\beta (\log 1/\eps)^2}{2 \log 1/q},\qquad \text{as $\eps\to 0$.}$$ \label{thm:a1sup} \end{thm}

\begin{proof} In view of the trivial fact that $M\leq S$, we only have to show an upper bound for $\pr{M \leq \eps}$. Note that $$\log \pr{M \leq \eps} = \sum_{n=0}^\infty \log\pr{q^n \wast \leq \eps} \leq \sum_{n=0}^N \log\pr{\wast \leq \eps q^{-n}},$$ for any $N\geq 0$. We choose the largest $N\in \N$ such that $\eps q^{-N} \leq 1$, i.e., as $\eps\to 0$, $N\sim \log (1/\eps)/\log (1/q)$. Then, by assumption, the last term is less than $$\sum_{n=0}^N \log (c \eps q^{-n})^\beta =  \beta (-\log 1/\eps + \log c)(N+1)+ \sum_{n=0}^N \beta n \log 1/q,$$ for some constant $c>0$ depending only on the lower tail of $\wast$. A short calculation shows that the asymptotic order of this term, as $\eps\to 0$, is indeed $-\beta (\log 1/\eps)^2 / (2 \log 1/q)$, as asserted. \end{proof}

We finish with the case considered in Theorem~\ref{thm:a2}, where $\wast$ has exponentially little mass near the origin. Here, the small deviation rates for $S$ and $M$ are the same, however, the small deviation constants differ (compare the constants in~(\ref{eqn:mexca1}) and~(\ref{eqn:mexca2})).

\begin{thm} \label{thm:a2sup} Let $-\log \pr{ \wast \leq \eps} \sim  K \eps^{-\gamma}$, as $\eps\to 0$, for some $\gamma>0$ and $K>0$. Then  \begin{equation}-\log \pr{M \leq \eps} \sim  \frac{{K}}{1-q^{\gamma}}\, \eps^{-\gamma},\qquad \text{as $\eps\to 0$.} \label{eqn:mexca2} \end{equation} \end{thm}

\begin{proof} Let $\tau\in]0,1[$. By assumption, there is a $0<\delta=\delta(\tau)<1$ such that for all $\eps<\delta$, \begin{equation} (1-\tau) K\eps^{-\gamma}\leq -\log \pr{ \wast \leq \eps} \leq (1+\tau)  K\eps^{-\gamma}. \label{eqn:eosp}\end{equation}

Let $\eps<\delta$. We choose the maximal $N_1\in\N$ such that $\eps q^{-N_1}\leq \delta$ and the maximal $N_2\in\N$ such that  $\eps q^{-N_2}\leq 1$. Note that \begin{multline}-\log \pr{M \leq \eps} =\sum_{n=0}^{N_1}- \log\pr{q^n \wast \leq \eps} \\ + \sum_{n=N_1+1}^{N_2}- \log\pr{q^n \wast \leq \eps} + \sum_{n=N_2+1}^\infty -\log\pr{q^n \wast \leq \eps}. \label{eqn:lasth}\end{multline}
Let us look at the first term in (\ref{eqn:lasth}). Using (\ref{eqn:eosp}) we estimate it from above by$$\sum_{n=0}^{N_1}- \log\pr{ \wast \leq q^{-n}\eps} \leq (1+\tau) K \eps^{-\gamma} \sum_{0\leq n \leq \frac{\log \delta/\eps}{\log 1/q}} (q^\gamma)^n \sim \frac{(1+\tau) K \eps^{-\gamma} }{1-q^\gamma},$$ as $\eps\to 0$. Analogously, the respective lower bound (with $(1-\tau)$) follows.

Let us look at the second term in (\ref{eqn:lasth}). We use the assumption to estimate $$\sum_{n=N_1+1}^{N_2}- \log\pr{\wast \leq q^{-n}\eps} \leq \sum_{n=N_1+1}^{N_2}  C \eps^{-\gamma} q^{n \gamma} = C \eps^{-\gamma}\, \frac{q^{(N_1+2)\gamma}-q^{(N_2+1)\gamma}}{1-q^\gamma},$$ which holds for some constant $C>0$ only depending on the lower tail of $X$. By the choice of $N_1$ and $N_2$, the last term can be estimated from above by $$C \eps^{-\gamma}\, \frac{(\eps/\delta)^\gamma q^\gamma-\eps^\gamma q^\gamma}{1-q^\gamma} = C q^\gamma \eps^{-\gamma}\, \eps^\gamma \,\frac{ \delta^{-\gamma}-1}{1-q^\gamma}.$$ 

Finally we come to the third term in (\ref{eqn:lasth}). Note that $$\sum_{n=N_2+1}^\infty -\log\pr{q^n \wast \leq \eps} \approx \sum_{q^n < \eps} \pr{q^n \wast > \eps} \int_{q^{n+1}}^{q^n} \frac{q^{-n}}{1-q}\, \d x,$$ where we used that $\log x \approx x-1$ for $x\to 1$ and $\pr{X\leq 1}>0$. When $\eps\to 0$, the last term is of order $$\approx  \int_0^{\eps} \pr{x \wast > \eps}\, \frac{\d x}{x} = \int_0^{\eps} \E \indi{X>\eps/x}\, \frac{\d x}{x}  = \E  \int_{0}^{\eps} \indi{X>\eps/x} \, \frac{\d x}{x}.$$ This equals $\E \log \max(\wast,1)$, which is the finite constant from (\ref{eqn:starns}).

Therefore, putting the three estimates together, $$-\log \pr{M \leq \eps} \lesssim \frac{(1+\tau) K}{1-q^\gamma}\, \eps^{-\gamma} ,$$ as $\eps\to 0$, for all $\tau>0$. Letting $\tau\to 0$ gives the upper bound. The lower bound is established in exactly the same way. \end{proof}

\section{Tauberian theorem and final comments} \label{sec:final} Let us finally provide the mentioned Tauberian-type arguments. The proof is along the same lines of the proof of Theorem~4.12.9 in \citeasnoun{bgt}. We refer e.g.\ to \citeasnoun{br1}, Lemma~6.1, and the reference mentioned there.

\begin{lem} \label{lem:app} Let $S$ be any non-negative random variable and let $\gamma\geq 1$ and $K>0$. Then \begin{equation} -\log \E e^{-\lambda S} \sim K (\log\lambda)^\gamma,\quad \text{as $\lambda\to \infty$,}\label{eqn:app1} \end{equation} holds if and only if \begin{equation} -\log \pr{ S\leq \eps} \sim K (\log 1/\eps)^\gamma,\quad\text{as $\eps\to 0$.}\label{eqn:app22} \end{equation}\end{lem}

\begin{rem} \label{rem:simplif} Finally, let us comment on whether $\sigma_n\sim q^n$ can be replaced by $\sigma_n=q^n$. Several authors have investigated this question, cf.\ \citeasnoun{licomp}, \citeasnoun{ght}, \citeasnoun{gl}, \citeasnoun{phd}, or \citeasnoun{aurzada1}. In fact, using the technique from the proof of Lemma~4.1 in \citeasnoun{aurzada1} it is easy to see that in the situation of Theorems~\ref{thm:a3}, \ref{thm:a1}, \ref{thm:a3sup}, and~\ref{thm:a1sup} the assumption $\sigma_n=q^n$ can be replaced w.l.o.g.\ by $\sigma_n=q^n$.

Contrary to this, for Theorems~\ref{thm:a2} and~\ref{thm:a2sup} it is not sufficient to assume $\sigma_n\sim q^n$ as one can see easily from the stable sum in Example~\ref{exa:stable} or the proof of Theorem~\ref{thm:a2sup}. Nevertheless, here the proof of Theorem~\ref{thm:a2sup} also serves in the general case. In the case of the sum however it is not clear how to determine the correct constant.
\end{rem}

\medskip
\noindent {\bf Acknowledgements:} I would like to thank A.~A.~Borovkov and P.~S.~Ru\-zan\-kin for sending me their preprints (\citeasnoun{br1} and \citeasnoun{br2}) and for valuable comments. This research was supported by the DFG Research Center \textsc{Matheon} ``Mathematics for key technologies'' in Berlin.

\hypertarget{vd}{~}\pdfbookmark[1]{References}{vd}
\bibliographystyle{alpha}

\end{document}